\documentclass[11pt]{article}
\usepackage[scale=.675,centering]{geometry}
\usepackage{amssymb,amsmath}
\usepackage{latexsym}
\usepackage[round]{natbib}
\usepackage{fleqn,latexsym}
\usepackage{wrapfig}
\usepackage{hyperref}

\renewcommand{\baselinestretch}{1.0}
\usepackage{amsthm}

\title{Hanf Locality and Invariant Elementary Definability
\thanks{Henry Towsner is partially supported by NSF grant DMS-1340666.}
}

\author{%
Steven Lindell \\ Haverford College \and
Henry Towsner \\ University of Pennsylvania \and
Scott Weinstein \\ University of Pennsylvania}

\newtheorem{theorem}{Theorem}[section]
\newtheorem{lemma}[theorem]{Lemma}

\newtheorem{definition}[theorem]{Definition}

\newtheorem{remark}[theorem]{Remark}

\newtheorem{example}[theorem]{Example}

\newcommand{\pres}[1]{\ensuremath{\mathbb{#1}}}

\newcommand{\card}[1]{\mbox{$|#1|$}}
\newcommand{\strs}[1]{\mbox{$\mathfrak{#1}$}}

\renewcommand{\max}{\ensuremath{\mathsf{max}}}

\newcommand{\cb}[3]{\ensuremath{\mathsf{B}_{#3}({#1},{#2})}}
\newcommand{\rloct}[3]{\ensuremath{\sigma^{#1}({#2},{#3})}}
\newcommand{\loct}[2]{\ensuremath{\sigma^{\infty}({#1},{#2})}}
\newcommand{\gcb}[3]{\ensuremath{{#1}^{#3}({#2})}}
\newcommand{\dist}[3]{\ensuremath{\delta_{#3}({#1},{#2})}}

\newcommand{\rest}[2]{\ensuremath{(#1\!\upharpoonright\!#2)}}
\newcommand{\op}[2]{\mbox{$\langle #1,#2 \rangle$}}

\newcommand{\hre}[3]{\ensuremath{ #1 \approx_{#3} #2 }}
\newcommand{\bqe}[3]{\ensuremath{ #1 \equiv_{\mathfrak{#3}} #2 }}
\newcommand{\lqe}[3]{\ensuremath{ #1 \equiv_{#3} #2 }}
\newcommand{\he}[2]{\ensuremath{ #1 \approx #2 }}
\newcommand{\thr}[4]{\ensuremath{ #1 \approx_{#3,#4} #2 }}

\newcommand{\mathsc}[1]{{\normalfont\textsc{#1}}}

\usepackage[dvipdf]{color}
\usepackage{graphicx}
\definecolor{Blue}{rgb}{0.2,0.3,0.9}
\definecolor{Green}{rgb}{0.0,0.7,0.0}

\newtheorem{dispar}[equation]{}

\begin{document}
\maketitle

\begin{abstract}
We introduce some notions of invariant elementary definability which extend the notions of first-order order-invariant definability, and, more generally, definability invariant with respect to arbitrary numerical relations. In particular, we study invariance with respect to expansions which depend not only on (an ordering of) the universe of a structure, but also on the particular relations which determine the structure; we call such expansions \emph{presentations} of a structure. We establish two locality results in this context. The first is an extension of the original Hanf Locality Theorem to boolean queries which are invariantly definable over classes of locally finite structures with respect to \emph{elementary, neighborhood-bounded} presentations. The second is a non-uniform version of the Fagin-Stockmeyer-Vardi Hanf Threshold Locality Theorem to boolean queries which are invariantly definable over classes of bounded degree structures with respect to elementary, neighborhood-bounded, \emph{local} presentations.

\end{abstract}




\section{Introduction}

Locality theorems characterize various senses in which properties of a relational structure that are definable in a particular way depend only on ``local information'' about the structure, typically, isomorphism types of finite radius neighborhoods in the Gaifman-graph of the structure. Such theorems provide simple and perspicuous proofs of descriptive lower bound results, that is, proofs that various combinatorial properties  cannot be defined with specified logical resources (see,\ \cite{Ga82},\ \cite{FSV95},\ \cite{Nurmonen96},\ \cite{HLN99}). They are also central to the proofs of various preservation theorems over classes of finite structures (see\ \cite{ADK06},\  \cite{ADG08},\ \cite{Dawar10}) and of numerous algorithmic metatheorems (see \cite{gk2011}). In this paper, we explore  Hanf locality of boolean queries which are invariantly definable on classes of (locally) finite structures. 

We introduce some notions of invariant elementary definability which extend the notions of first-order order-invariant definability, and, more generally, definability invariant with respect to arbitrary numerical relations; we study invariance with respect to expansions which depend not only on (an ordering of) the universe of a structure, but also on the particular relations which determine the structure.  We call such expansions \emph{presentations} of a structure. Examples of presentations are breadth-first and depth-first traversals of finite graphs, and local orders of locally finite graphs. We establish two locality results in this context. The first, Theorem 
\ref{nbpresinvHanf-thm}, is an extension of the original Hanf Locality Theorem to boolean queries which are invariantly definable over classes of locally finite structures with respect to \emph{elementary, neighborhood-bounded} presentations. The second, Theorem \ref{locpresthreshloc-thm}, is a non-uniform version of the Fagin-Stockmeyer-Vardi Hanf Threshold Locality Theorem to boolean queries which are invariantly definable over classes of bounded degree structures with respect to elementary, neighborhood-bounded, \emph{local} presentations.
\paragraph{Related Work}

The study of locality  properties of first-order logic began with \cite{Han65}, who proved that relational structures which are fully locally equivalent up to a countable multiplicity threshold are elementarily equivalent.  \cite{Ga82} analyzed locality of relations definable within a single structure and established local normal forms for first-order formulas that have application to descriptive lower bounds over classes of finite structures as well as infinite structures.  \cite{FSV95} established a refinement of the locality theorem of \cite{Han65} which shows that for each first-order sentence $\theta$, there is a locality radius $r_\theta$, depending only on the quantifier rank of $\theta$, such that structures locally equivalent up to radius $r_\theta$ and up to a countable multiplicity threshold, are $\theta$-equivalent. Moreover, they showed that over structures of bounded degree $d$, the multiplicity threshold depends exponentially on $d$.  \cite{GS00} address the question of locality of queries which are not first-order definable over finite structures. In particular, they establish Gaifman locality of queries which are order-invariant first-order definable over finite structures. \cite{HLN99} provides an extensive and detailed study of the relations among several notions of locality and establishes locality results for several extensions of first-order logic, among them $\mathsc{FO+COUNT}$. The paper also provides a counterexample to Gaifman locality of queries definable order-invariantly over finite structures by  $\mathsc{FO+COUNT}$-formulas, and suggests the possibility that such a locality result might hold non-uniformly in the size of a structure, that is, when the radius of locality is allowed to depend not only on the quantifier rank of a formula defining the query but also on the cardinality of the structure over which it is evaluated. \cite{AMSS12} establishes just such a non-uniform form of Gaifman locality for queries definable in $\mathsc{FO}$ by formulas invariant with respect to arbitrary numerical predicates over finite structures, that is, predicates whose interpretation depends only on an arbitrary linear order of the universe; in particular, they demonstrate Gaifman locality for any query first-order definable invariantly with respect to arbitrary numerical predicates, where the radius of locality depends polylogarithmically on the size of the structure over which it is evaluated. \cite{AMSS12} also establishes a similar non-uniform form of Hanf locality for queries definable by first-order formulas invariant with respect to arbitrary numerical predicates over string structures, again with radius of locality depending polylogarithmically on the size of the structure. The question of such a Hanf locality result for arbitrary finite structures remains open.  In this paper we explore Hanf locality of queries which are definable over some ``tame'' classes of structures (locally finite structures and bounded degree structures) invariantly with respect to certain ``tame'' presentations (neighborhood bounded and local presentations). We consider presentations of a structure other than built-in numerical relations, in that they depend (elementarily) on relations of the structure other than an arbitrary ordering of its universe. In this context, we establish a Hanf type locality result for invariant definability over locally finite structures, and a Hanf threshold type result for invariant definability over bounded degree structures. The latter result is ``non-uniform'', in the sense that the radius of locality depends not only on the defining formulas, but also on the degree of structures over which it is evaluated. Our threshold theorem applies to boolean queries which are invariantly definable via local-orders over structures with Gaifman graphs of bounded degree; invariant definability via orders and local-orders over ``tame'' classes of finite graphs, such as graphs of bounded degree and bounded tree-width, has been studied in \cite{benseg} where certain ``collapse'' results were established.    
\section{Preliminaries}

The next definition formulates some topological notions requisite for describing locality precisely. We write $\strs{K}_\tau$ ($\strs{F}_\tau$) for the collection of (finite) structures of finite relational similarity type $\tau$, that is, $\tau$ consists of finitely many relation and constant symbols.
\begin{definition}\label{D:balls}
Let $G$ be a simple graph, $a,b\in G$ and $r\geq 0$.
\begin{enumerate} 
\item We write $\dist{a}{b}{G}$ for the distance between $a$ and $b$ in $G$.
\item $\cb{G}{a}{r}=\{c\in G\mid \dist{a}{c}{G}\leq r\}$, for $r\in\omega$. $\cb{G}{a}{\infty}=\bigcup_{r\in\omega}\cb{G}{a}{r}$.
\item Let $A\in\strs{K}_\tau$. The \emph{Gaifman-graph} of $A$ (denoted $G_A$) is the simple graph whose node set is $A$ and whose edge relation holds between a pair of nodes $a,b$ with $a\neq b$ if and only if for some $R\in\tau$ and tuple $\overline{c}\in R^A$ there are $i,j$ with $c_i = a$ and $c_j =b$.
\item We say a structure $A$ has \emph{ degree bounded by $d$} (\emph{finite degree}) if and only if  for every $a\in A$ the degree of $a$ in $G_A$ is at most $d$ (degree of $a$ in $G_A$ is finite). We write $\strs{K}^d_\tau$ ($\strs{F}^d_\tau$) for the collection of (finite) $\tau$-structures with degree bounded by $d\in\omega$, and $\strs{K}^{<\omega}_\tau$ for the collection of $\tau$-structures of finite degree.
 We will suppress the subscript $\tau$ when it is clear from the context. 
\item Let $r\in\omega$ or $r=\infty$, and let  $a\in A$. 
\[\gcb{A}{a}{r}= \op{\rest{A}{\cb{G_A}{a}{r}}}{a}.\] 
The structure $\gcb{A}{a}{r}$ is  called the pointed $r$-neighbor-hood of $a$ in $A$. Note that $\gcb{A}{a}{\infty}$ is the substructure of $A$ induced by the connected component of $a$ in $G_A$.
\end{enumerate}
\end{definition}
$\strs{K}^{<\omega}_\tau$ is also known as the class of \emph{locally finite} $\tau$-structures, because $A\in\strs{K}^{<\omega}_\tau$ if and only if for every $a\in A$ and every $r\in\omega$, $\cb{G_A}{a}{r}$ is finite. Note that if $A$ is locally finite, then for every $a\in A$, $\cb{G_A}{a}{\infty}$ is either finite or countably infinite. 

Let $t\in \omega\cup\{\omega\}$. We say sets $X$ and $Y$ are equipollent up to threshold $t$ (written $X\sim_t Y$) if and only if either $\card{X} = \card{Y}$ or $\card{X},\card{Y}\geq t$.

Let $\strs{J}\subseteq\strs{K_\tau}$ and $A,B\in\strs{K}_\tau$. We write $\bqe{A}{B}{J}$ if $A$ and $B$ agree on \strs{J}, that is, $(A\in\strs{J} \iff B\in\strs{J})$, and similarly,
$\lqe{A}{B}{\theta}$ for $(A\models\theta \iff B\models\theta)$, when $\theta$ is a first-order sentence in signature $\tau$. $A\equiv B$ ($A$ is elementarily equivalent to $B$) if and only if $\lqe{A}{B}{\theta}$, for all first-order sentences $\theta$.

\begin{definition}
Suppose $A,B\in\strs{K}_\tau$ and $\strs{K}\subseteq\strs{K}_\tau$. 
\begin{enumerate}
\item 
Let $r\in\omega$ and $t\in\omega\cup\{\omega\}$. $A$ is \emph{Hanf threshold $r,t$-equivalent to $B$} (written $\thr{A}{B}{r}{t}$) if and only if for every pointed $\tau$-structure $C$, 
 \[\{a\in A\mid \gcb{A}{a}{r}\cong C\} \sim_t \{b\in B\mid \gcb{B}{b}{r}\cong C\},\]
that is, each isomorphism type of pointed $r$-neighbor-hood is realized the same number of times in $A$ and $B$ up to threshold $t$. 
 \item 
 A class of $\tau$-structures $\strs{J}$ is \emph{Hanf $r,t$-local on $\strs{K}$} if and only if for all $A,B\in \strs{K}$, if $\thr{A}{B}{r}{t}$, then $\bqe{A}{B}{J}$,
 that is, $\strs{J}$ is closed under $r,t$-equivalence with respect to structures in \strs{K}.
\item 
Let $r\in\omega$. $A$ is \emph{Hanf $r$-equivalent to $B$} (written $\hre{A}{B}{r}$) if and only if $\thr{A}{B}{r}{\omega}$. (Thus, if $A$ and $B$ are countable and Hanf $r$-equivalent, there is an isomorphism preserving bijection between their pointed $r$-neighborhoods.)
\item 
A class of $\tau$-structures $\strs{J}$ is \emph{Hanf $r$-local on $\strs{K}$} if and only if for all $A,B\in \strs{K}$, if $\hre{A}{B}{r}$, then $\bqe{A}{B}{J}$, that is, $\strs{J}$ is closed under $r$-equivalence with respect to structures in \strs{K}.
\item 
$A$ is \emph{Hanf equivalent to $B$} (written $\he{A}{B}$) if and only if for all $r\in\omega$, $\hre{A}{B}{r}$, that is, $A$ and $B$ are fully locally equivalent up to threshold $\omega$.
\item 
A class of $\tau$-structures $\strs{J}$ is \emph{Hanf local on $\strs{K}$} if and only if for all $A,B\in \strs{K}$, if $\he{A}{B}$, then $\bqe{A}{B}{J}$, that is, $\strs{J}$ is closed under Hanf equivalence with respect to structures in \strs{K}. 
\end{enumerate} 
\end{definition}
\section{Hanf Locality}

In general, elementary boolean queries are not Hanf local on arbitrary structures. For example, the simple graphs $G$ and $H$ consisting of a single countable clique and the disjoint union of two countable cliques are distinguished by a first-order sentence, yet $\he{G}{H}$. On the other hand, \cite{Han65} showed that every elementary class, that is, a class defined by a set of first-order sentences, is Hanf local on the class of locally finite structures.
\begin{theorem}[Hanf]\label{Han65-thm}
For every $\tau$ and $A,B\in\strs{K}^{<\omega}_\tau$, if $\he{A}{B}$, then $A\equiv B$.
\end{theorem}
Note that for finite structures $A$ and $B$, $\he{A}{B}$ if and only if $A\cong B$; thus Hanf equivalence is not a useful tool for analyzing the expressive power of logical languages over structures which are actually finite, as opposed to merely locally finite. \cite{FSV95} proved that classes of structures defined by single first-order sentences $\theta$ over $\strs{F}_\tau$ are Hanf $r$-local, for large enough $r$, depending only on the quantifier rank of $\theta$. 
\begin{theorem}[Fagin, Stockmeyer, \& Vardi]\label{FSV95-thm}
For every $\tau$ and for every first-order sentence $\theta$ with signature $\tau$, there is an $r\in\omega$ such that for all $A,B\in\strs{F}_\tau$, if $\hre{A}{B}{r}$, then $A\equiv_\theta B$.
\end{theorem}
Theorem \ref{FSV95-thm} provides a valuable technique for establishing that many well-known combinatorial properties of finite structures are not first-order definable. 
\cite{FSV95} derived Theorem \ref{FSV95-thm} from the following threshold locality theorem for first-order logic.
\begin{theorem}[Fagin, Stockmeyer, \& Vardi]\label{FSV95-threshold-thm}
For every $\tau$ and for every first-order sentence $\theta$ with signature $\tau$, there is an $r\in\omega$ such that for all $d\in\omega$, there is a $t\in\omega$, such that for all $A,B\in\strs{F}^d_\tau$, if $\thr{A}{B}{r}{t}$, then $A\equiv_\theta B$.
\end{theorem}
Here too, the locality radius may be chosen to depend only on the quantifier rank of $\theta$, whereas the threshold $t$ depends also on the degree bound $d$ on the Gaifman graphs of the structures. It is easy to verify that the argument of \cite{FSV95} suffices to establish Theorems \ref{FSV95-threshold-thm} for arbitrary structures of bounded degree, not just finite structures, that is, for $\strs{K}^d_\tau$ in place of $\strs{F}^d_\tau$. \cite{Libkin04} presents a direct proof of Theorem \ref{FSV95-thm}, with generalizations to more powerful logical languages, and much illuminating discussion of locality and its applications.

\section{Hanf Locality of Queries Invariantly Definable over Locally Finite Structures}

In logic, a mathematical object, such as a simple graph, is typically identified with an abstract relational structure: in the case of a simple graph, this consists of a set of vertices paired with an irreflexive and symmetric edge relation between them. But this identification is hardly universal throughout mathematics and computer science. For example, a graph is often presented as a drawing of nodes in the plane together with arcs joining them to indicate the edges, or as an adjacency matrix whose axes are labelled by the vertices in some arbitrary order and whose entries indicate the presence or absence of an edge. In this section, we will focus attention on the connection between an abstract structure and such presentations of it. In particular, we will investigate invariant elementary definability over classes of finite and of locally finite structures with respect to various such schemes of presentation. 

A special case of presentation invariant definability has been studied in the context of database theory. The ``Data Independence Principle'' (see \cite{AHV}) states that the results of a database query must be independent of (that is, invariant with respect to) the arbitrary order via which the entries in the database are encoded. The importance of this condition has given rise to sustained investigation of the properties of queries that are (linear) order-invariant definable over finite structures in various logical languages that underly database query languages. Linear orders are a very special kind of presentation of a structure that involves only the underlying universe of the structure while ignoring its relational component. Presentations of this general character are often called ``built-in numerical relations,'' and locality of queries invariantly definable with respect to such presentations has been studied intensively in finite model theory (see, for example, \cite{GS00} and \cite{AMSS12}). In this section we will introduce more general schemes of presentation.   
\begin{definition}
Let $A\in \strs{K}_\tau$ and $R$ be a relation symbol (of some fixed but unspecified arity). 
\begin{enumerate}
\item
An \emph{$R$-presentation} of $A$ is an expansion of $A$ to a structure $A^*$ of similarity type $\tau\cup\{R\}$. 
\item
An \emph{$R$-presentation scheme} $\pres{P}$ for a class of structures $\strs{K}\subseteq\strs{K}_\tau$ is a collection of $R$-presentations of members of $\strs{K}$ that contains at least one presentation of each member of $\strs{K}$. We sometimes write $A^*\pres{P}A$ to abbreviate $A^*\in \pres{P}$ and $A^*$ is an $R$-presentation of $A$.
\item
An  $R$-presentation scheme $\pres{P}$ of a collection of structures $\strs{K}\subseteq\strs{K}_\tau$  is an \emph{elementary presentation scheme} if and only if there is a first order sentence $\theta$ such that for every $A\in\strs{K}$ and every $R$-presentation $A^*$ of $A$, 
\[A^*\in \pres{P} \iff A^*\models \theta.\] 
\end{enumerate}  
\end{definition}

\begin{example}
Let \strs{G}\ be the collection of finite simple graphs.
\begin{enumerate}
\item For every $G\in \strs{G}$, $\op{G}{<}\in\pres{L}(\strs{G})$ if and only if $<$ linearly orders the vertices of $G$. We call this the \emph{linear presentation} of finite graphs.
\item For every $G\in \strs{G}$, $\op{G}{<}\in\pres{T}(\strs{G})$ if and only if $<$ is a \emph{traversal} of $G$. (Recall that $<$ is a traversal of a graph $G$ just in case it is a linear ordering of the vertices of $G$ with the additional property that the connected components of $G$ occupy disjoint intervals, and that in each such interval, every node but the first has a preceding neighbor. It is well-known (see \cite{CorneilK08}) that $<$ is a traversal of a graph $G$ if and only if it satisfies the following first-order condition:
for all $a,b,c\in G$, if $a<b<c$ and $Eac$, then there is a $d\in G$ such that $d<b$ and $Edb$.) We call this the \emph{traversal presentation} of finite graphs. 
\item A ternary relation $O$ on the vertex set of a graph $G$ \emph{locally orders} $G$ if and only if for every $a\in G$, the binary relation $\{\op{b}{c}\mid Oabc\}$ is a strict linear ordering of the $G$-neighbors of $a$. For every $G\in\strs{G}$, $\op{G}{O}\in\pres{O}(\strs{G})$ if and only if $O$ locally orders $G$. We call this the \emph{local order presentation} of finite graphs.
\end{enumerate}
\end{example}
It is easy to see that all these examples are elementary presentation schemes.
Moreover, each of these examples can be extended to collections of finite $\tau$-structures via their Gaifman graphs. We next explore the power of presentations to reveal information about the structures they present. We characterize this power in terms of presentation-invariant elementary definability.
\begin{definition}\label{presinv-defn}
Let \pres{P}\ be an elementary $R$-presentation scheme for a class of structures $\strs{K}\subseteq\strs{K}_\tau$.
\begin{enumerate}
\item A first order sentence $\theta$ in the signature $\tau\cup\{R\}$ is \pres{P}-invariant over \strs{K}\ if and only for every $A\in\strs{K}$, 
$(\exists A^*) (A^*\pres{P}A\wedge A^*\models\theta)$ if and only if\\
 $(\forall A^*)(A^*\pres{P}A\rightarrow A^*\models\theta)$.
\item $\strs{J}\subseteq\strs{K}$ is \pres{P}-invariant elementary on \strs{K}\ if and only if there is a first order sentence $\theta$ that is \pres{P}-invariant on \strs{K}\ and for every $A\in\strs{K}$
\[A\in\strs{J}\ \ \iff\ \  (\exists A^*)(A^*\pres{P}A\wedge A^*\models\theta).\]
\end{enumerate}
\end{definition}
It follows from the the Beth Definability Theorem that a boolean query is \pres{P}-invariant elementary over the class of arbitrary structures if and only if it is elementary, for any elementary presentation scheme \pres{P}. On the other hand, a well-known example due to Gurevich (first exposed in Exercise 17.27 of \cite{AHV}), namely the collection of \emph{finite} boolean algebras with an even number of atoms, shows that there are queries \pres{L}-invariant elementary over the collection of finite structures that are not elementary. Observe that this collection is also \pres{O}-invariant elementary. It is easy to see that  connectivity is \pres{T}-invariant elementary over the the class of finite simple graphs. On the other hand, \cite{Gurevich84} shows that connectivity is not \pres{L}-invariant elementary.

As the example of traversals demonstrates, we will need to consider ``tame'' classes of elementary presentations in order to establish locality of invariantly definable boolean queries. The first tameness condition we introduce, neighborhood boundedness, excludes presentations which themselves have ``global reach'', that is, neighbors of a point in the (Gaifman graph) of the presentation of a structure cannot be arbitrarily far way from that point in the (Gaifman graph) of the structure itself.  
\begin{definition}\label{nb-pres-def}
  Suppose \pres{P}\ is a presentation scheme for the class of all locally finite 
$\tau$-structures  $\strs{K}^{<\omega}_\tau$. 
\begin{enumerate}
\item
We say $\mathbb{P}$ is \emph{neighborhood bounded} if and only if there is a constant $\nu$ (called the \emph{neighborhood expansion factor} for \pres{P}) such that for every $A\in\mathfrak{K}^{<\omega}_\tau$, for every $A^*$ such that $A^*\mathbb{P}A$, and
  for every $a\in A$,  
 $\cb{G_{A^*}}{a}{1}\subseteq\cb{G_{A}}{a}{\nu}$.
\item
We say $\mathbb{P}$ is \emph{degree bounded} if and only if there is a function $g:\omega\mapsto\omega$ (called a \emph{degree bound} for \pres{P}) such that for every $d,A,A^*$, if $A\in\strs{K}^d_\tau$ and $A^*\pres{P}A$, then $A^*\in \strs{K}^{g(d)}_{\tau\cup R}$.
\end{enumerate}
\end{definition}
Observe that if \pres{P} is neighborhood bounded with neighborhood expansion factor $\nu$ then for every $A\in\mathfrak{K}^{<\omega}_\tau$, every $A^*$ such that $A^*\pres{P}A$, every $a\in A$, and every $r\in\omega$, $\cb{G_{A^*}}{a}{r}\subseteq\cb{G_{A}}{a}{\nu\cdot r}$. Note also, that if \pres{P}\ is neighborhood bounded, then for every $A\in\mathfrak{K}^{<\omega}_\tau$ and for every $A^*$ such that $A^*\mathbb{P}A$, $A^*\in\mathfrak{K}^{<\omega}_{\tau\cup\{R\}}$. 

Note \pres{O}, the presentation scheme of locally ordered finite graphs, is neighborhood bounded (with neighborhood expansion factor $\nu=2$), while \pres{L}, the presentation scheme of linearly ordered finite graphs, is not (the Gaifman graph of a linearly ordered set is a clique). As the results of \cite{GS00} and \cite{AMSS12} demonstrate, the failure of a presentation to be neighborhood bounded does not preclude locality of boolean queries invariantly definable with respect to it. 
\begin{remark}\label{degbd-remark}
Observe that every neighborhood bounded presentation \pres{P}\ is degree bounded. In particular, suppose that $\nu$ is the neighborhood expansion factor of \pres{P}, and that $A\in\strs{K}^d_\tau$ and $A^*\pres{P}A$. For any point $a\in A^*$, the degree of $a$ is just the size of the neighborhood $\cb{A^*}{a}{1} - 1$. But, this is bounded by the size of the neighborhood $\cb{A}{a}{\nu} - 1$ which is bounded by $d^{\nu}$.
\end{remark}

A crucial tool in the proof of Theorem \ref{nbpresinvHanf-thm} will be the use of the \emph{local type} of a point, that is, the collection of formulas with all quantifiers bounded to some $r$-neighborhood of the point. Note that for any finite relational signature $\tau$, there is a first-order formula $\delta_{\tau}^r(x,y)$ such that for every $\tau$-structure $A$ and $a,b\in A$, 
$A\models \delta_{\tau}^r[a,b]$ if and only if the distance between $a$ and $b$ in $G_A$ is at most $r$. If $\varphi(x_1, \ldots, x_k)$ is a formula, we write $\varphi^r(x_1, \ldots, x_k)$ for the formula which results from $\varphi(x_1, \ldots, x_k)$ by relativizing all its quantifiers $(Qy)$ to the formula $\delta_{\tau}^r(x_1,y)\vee\ldots\vee\delta_{\tau}^r(x_k,y)$. We call such a formula an \textit{$r$-local formula}. If $\overline{a}$ is a $k$-tuple of elements of $A$, the \textit{$r$-local type} of $\overline{a}$ in $A$ is the set of formulas $\varphi^r(\overline{x})$ such that $A(\overline{a})\models\varphi^r(\overline{a})$. Note that for structures $A\in\strs{K}^{<\omega}_\tau$ and $a\in A$, the $r$-local type of $a$ in $A$ is determined by a single $r$-local formula which characterizes the \textit{finite} structure $\gcb{A}{a}{r}$ up to isomorphism. We write $\rloct{r}{a}{A}(x)$ for the $r$-local type of $a$ in $A$. The \textit{local type} of $\overline{a}$ in $A$ is the union of the $r$-local types of $\overline{a}$ in $A$, for all $r\in\omega$. We write $\loct{a}{A}$ for the local type of $a$ in $A$. It is easy to see that the connected component of a point determines its local-type, since any quantifiers in an $r$-local formula range only over the connected component; Theorem \ref{nbpresinvHanf-thm} exploits the fact that the converse holds in locally finite structures. 

The next theorem establishes that the original form of Hanf's Theorem extends to boolean queries that are first-order invariantly definable with respect to neighborhood bounded presentations, that is, if two structures are Hanf equivalent for every finite radius up to countable threshold, then they agree on every such query.
\begin{theorem}\label{nbpresinvHanf-thm}
  Let $\mathbb{P}$ be an elementary neighborhood bounded presentation and $\mathfrak{J}\subseteq\mathfrak{K}_\tau$ a $\mathbb{P}$-invariant elementary boolean query on $\mathfrak{K}_\tau^{<\omega}$.  If $A,B\in\mathfrak{K}^{<\omega}_\tau$ and $A\approx B$, then $\bqe{A}{B}{J}$.
\end{theorem}

\begin{proof}
To show this, we need to observe that there is a ``universal'' structure $U$ in the Hanf-class of $A$, that is, the class of structures $C$ with $\he{A}{C}$.  Note that, in a locally finite structure $A$, the local type of a point $a$ determines its connected component (in $G_A$, the Gaifman graph of $A$). Indeed, for each $r\in\omega$, $\rloct{r}{a}{A}$ determines the isomorphism type of the finite structure $\gcb{A}{a}{r}$. Thus, if $A$ and $B$ are locally finite, with $\rloct{\infty}{A}{a}=\rloct{\infty}{b}{B}$, there is a finitely branching tree of isomorphisms from the neighborhoods $\gcb{A}{a}{r}$ onto the neighborhoods $\gcb{B}{b}{r}$, and thence, by the K\"{o}nig Infinity Lemma, an isomorphism from the pointed component  $\gcb{A}{a}{\infty}$ onto the pointed component $\gcb{B}{b}{\infty}$.  Therefore, in particular, if $A\approx B$ and $a\in A$ is a point such that, for some $r$, $A^r(a)$ appears finitely many times in $A$, then $B$ must contain connected components isomorphic to the connected component of $a$ with exactly the same multiplicity as they occur in $A$.

Consider the connected components of $A$; we may divide them into two kinds: those components which appear finitely often, and those which appear infinitely often.  We define the universal structure $U$ to consist of each of the components of $A$ which appears finitely many times in $A$ (and the same number of each such component as in $A$) together with countably many of each component such that every finite-radius neighborhood of this component appears infinitely often in $A$.  (Note that such a component might not, itself, appear in $A$.)  Note that $U$ would be the same if we constructed it based on $B$ instead. Therefore, by the assumption that $\strs{J}$ is a $\mathbb{P}$-invariant elementary boolean query, it suffices to show that, given a presentation $A^*\mathbb{P}A$, there is a presentation $U^*\pres{P}U$ so that $\he{A^*}{U^*}$.  

We construct $U^*$ as follows.  First, we simply copy each component of $A^*$ to an identical component of $U^*$ (which can be done since there is an embedding of $A$ into $U$).  Next, consider some component $C$ of $U$ which has not been assigned a presentation.  It suffices to assign a \pres{P}-presentation $C^*$ to $C$  such that each $r$-neighborhood in the presentation is present infinitely often in $A^*$.  Take any point $u\in C$ and let $\nu$ be the neighborhood expansion factor of \pres{P}. Note first that for every $r$, \gcb{C}{u}{\nu\cdot r} appears infinitely often in $A$,  and, since $A^*$ is locally finite, there are, up to isomorphism, only finitely many presentations of \gcb{C}{u}{\nu\cdot r}, so at least one of these presentations appears infinitely often in $A^*$. It follows at once that we may inductively choose a sequence of \pres{P}-presentations \gcb{C^*}{u}{r} such that for every $r$, $\gcb{C^*}{u}{r}\subseteq\gcb{C^*}{u}{r+1}$ and \gcb{C^*}{u}{r} occurs, up to isomorphism, infinitely often in $A^*$. Finally, we let $C^* = \bigcup_{r\in\omega}\gcb{C^*}{u}{r}$.
\end{proof}

Theorem \ref{nbpresinvHanf-thm} is only of interest in connection with \emph{infinite}, locally finite structures, such as arise in some approaches to software and hardware verification via model-checking; for finite $A$ and $B$, and ``logical'' presentations \pres{P}, if $A\approx B$, then $A\cong B$, thence if $A^*\mathbb{P}A$ and $B^*\mathbb{P}B$, then $A^*\cong B^*$. We do not know whether a Hanf locality result with finite threshold bound holds for boolean queries which are $\mathbb{P}$-invariant elementarily definable with respect to neighborhood bounded presentations in general. 
The next definition introduces a very natural class of \emph{local} presentation schemes, which allows us to establish a non-uniform Hanf threshold locality result for boolean queries invariantly definable both over  finite, and over locally finite, structures.   

\begin{definition}\label{local-pres-def}
  Suppose $\strs{K}\subseteq\strs{K}_\tau$ is closed under substructures and \pres{P}\ is a presentation scheme for \strs{K}. We say $\mathbb{P}$ is \emph{local} if and only if 
\begin{enumerate}
\item\label{substruc-local-cond} for all $A\in\strs{K}$, all $A^*$ with $A^*\pres{P}A$, and all finite $B\subseteq A$, $\rest{A^*}{B}\pres{P}B$, and
\item\label{disj-amalg-local-cond} for all $A\in\strs{K}$, all $B,C\subseteq A$, and all finite $B^*,C^*$, if $B\cap C=\emptyset$ and $B^*\pres{P}B$ and $C^*\pres{P}C$, then there is an $A^*$ such that $A^*\pres{P}A$ and $\rest{A^*}{B}=B^*$ and $\rest{A^*}{C}=C^*$.
\end{enumerate}
\end{definition}

We call condition (\ref{substruc-local-cond}) of Definition \ref{local-pres-def} \emph{localization}: if we restrict a presentation of a structure to a finite substructure, it is a presentation of that substructure. We call condition (\ref{disj-amalg-local-cond}) of Definition \ref{local-pres-def} \emph{disjoint local amalgamation}: presentations of  disjoint finite pieces of a structure may be combined in a presentation of the entire structure. It is easy to verify that 
both \pres{L}\ and \pres{O}\ are local. On the other hand, it is clear that \pres{T}, the presentation scheme of traversals of finite graphs, has neither localization nor disjoint local amalgamation. 
By modifying \pres{O}\ to the (elementary) presentation scheme of ``local circular successor'' on locally finite graphs, we have an example of a presentation with disjoint local amalgamation that lacks localization. On the other hand, the (elementary) presentation ``component coloring'', which requires each component to lie within, or be disjoint from, a given unary predicate, satisfies localization, but violates disjoint amalgamation. Note that both ``local circular successor'' and  ``component coloring'' are neighborhood bounded.

The next result shows that boolean queries invariantly definable with respect to local, neighborhood-bounded elementary presentations are Hanf threshold local on bounded degree structures. The result is non-uniform in the sense that the radius of locality, as well as the multiplicity threshold, depends on the degree of the structures over which the query is evaluated. 
\begin{theorem}\label{locpresthreshloc-thm}
Suppose  $\strs{K}\subseteq \strs{K}_\tau$ is closed under substructures and  \pres{P} is a local, neighborhood-bounded elementary $R$-presentation scheme for $\strs{K}$. If $\strs{J}\subseteq \strs{K}_\tau$ is a \pres{P}-invariant elementary boolean query on $\strs{K}$, then for all $d$, there are $r,t\in\omega$ such that $\strs{J}$ is Hanf $r,t$-threshold local on $\strs{K}\cap\strs{K}^d_\tau$.  
\end{theorem}
Note that Theorem \ref{locpresthreshloc-thm} is non-uniform in the degree bound $d$, that is, both the locality radius and the threshold multiplicity depend on $d$. This is in contrast to the Threshold Locality Theorem of \cite{FSV95} where the locality radius depends only on the quantifier rank of an elementary boolean query and not on the degree bound on the Gaifman graphs of a class of structures over which it is evaluated.

The bound on degree enters into the proof of Theorem \ref{locpresthreshloc-thm} in part via the fact that the class of structures of degree bounded by $d$ is wide, in the sense of \cite{ADG08}.
\begin{definition}
Let $\strs{K}\subseteq \strs{K}_\tau$ and $A\in\strs{K}$.
\begin{enumerate}
\item 
$X\subseteq A$ is \emph{$r$-scattered} (\emph{away from $Y\subseteq A$}) if and only if for all $a,a'\in X$, $\gcb{A}{a}{r}\cap\gcb{A}{a'}{r}=\emptyset$ (and for all $a'\in Y$, $\gcb{A}{a}{r}\cap\gcb{A}{a'}{r}=\emptyset$).
\item
\strs{K}\ is $p,q$-wide if and only if for every $r\in\omega$, there is a function $\zeta_r:\omega\mapsto\omega$, called the $p,q$-\emph{width growth factor} for \strs{K} with respect to $r$, such that for every $A\in\strs{K}$, $X_1,\ldots,X_p\subseteq A$, every $Z\subseteq A$ with $\card{Z}\leq q$, and $m\in\omega$, if $\card{X_i}\geq \zeta(m)$, for $1\leq i\leq p$, then there are sets $Y_i\subseteq X_i$ with $\card{Y_i}\geq m$, for $1\leq i\leq p$, and $\bigcup_{1\leq i\leq p}Y_i$ is $r$-scattered away from $Z$. 
\end{enumerate}
\end{definition}
The concept of a wide class, as defined in \cite{ADG08}, corresponds to $1,0$-wide. This notion was introduced, though not named, in \cite{ADK06}, where it was used to establish the homomorphism preservation theorem over classes of finite structures of bounded degree, which are, in addition, closed under substructures and disjoint unions. The notion, and its generalizations, almost wide and quasi-wide, were exploited in \cite{ADG08} and \cite{Dawar10} to establish both homomorphism preservation theorems and preservation with respect to extensions theorems over several ``tame'' classes of finite structures. Important connections between these notions of width and the theory of sparse graphs are expounded in \cite{Sparsity} where, among other things, it is established that a substructure closed class is wide if and only if it has bounded degree.

The proof of Theorem \ref{locpresthreshloc-thm} makes use of the following combinatorial lemma which is implicit in \cite{ADG08}.
\begin{lemma}\label{wide-lemma}
Suppose $\strs{K}$ has bounded degree, that is, there is a $d$ such that the maximum degree of $G_A$ for every $A\in\strs{K}$ is at most $d$. Then for every $n,q$, \strs{K} is $n,q$-wide.
\end{lemma}
    
\begin{proof}[Proof of Theorem \ref{locpresthreshloc-thm}]
Let $\strs{K}\subseteq \strs{K}_\tau$ be closed under substructures; we write $\strs{K} ^d$ for $\strs{K}\cap\strs{K}^d_\tau$. Let \pres{P}\ be a local, neighborhood-bounded elementary presentation scheme for $\strs{K}$, with neighborhood expansion factor $\nu$, and $\strs{J}\subseteq \strs{K}_\tau$ be a \pres{P}-invariant elementary boolean query on $\strs{K}$. By Remark \ref{degbd-remark}, we may also suppose that \pres{P}\ is degree bounded with degree bound $d^\nu$. 

By Definition \ref{presinv-defn}, we may fix
$\theta$ to be a first-order sentence \pres{P}-invariant on \strs{K}\ such that for every $A\in\strs{K}$
\[A\in\strs{J}\ \ \iff\ \  (\exists A^*)(A^*\pres{P}A\wedge A^*\models\theta).\]
We need to show that for all $d\in\omega$, there are  $r,t\in\omega$ such that for all $A,B\in\strs{K}^d$, if $\thr{A}{B}{r}{t}$, then $\bqe{A}{B}{J}$. Fix $d$, and let $d^*=d^\nu$.
By Theorem \ref{FSV95-threshold-thm}, we may choose $r_\theta\in\omega$ and  $t_{d^*}\in\omega$, such that for all $A^*,B^*\in\strs{K}^{d^*}_{\tau\cup R}$, if $\thr{A^*}{B^*}{r_\theta}{t_{d^*}}$, then $A^*\equiv_\theta B^*$. (Indeed, as we've emphasized above, $r_\theta$ depends only on (the quantifier rank) of $\theta$, and not on the degree bound $d^*$, but this observation does \emph{not} translate to uniformity of our result.)   

Suppose that for every $r,t\in\omega$, there are structures $A_{r,t}, B_{r,t}\in\strs{K}^{d}$ such that $\thr{A_{r,t}}{B_{r,t}}{r}{t}$, but $A_{r,t}\in\strs{J}$ and $B_{r,t}\not\in\strs{J}$; we derive a contradiction from this hypothesis, by showing that for large enough $r^*$ and $t^*$, if $\thr{A_{r^*,t^*}}{B_{r^*,t^*}}{r^*}{t^*}$, then there are presentations  $A^*_{r^*,t^*}\pres{P}A_{r^*,t^*}$ and $B^*_{r^*,t^*}\pres{P}B_{r^*,t^*}$ such that  $\thr{A^*_{r^*,t^*}}{B^*_{r^*,t^*}}{r_\theta}{t_{d^*}}$.

Let $s=\nu\cdot r_\theta$. Let $\iota_1,\ldots,\iota_k$ be a list of all the isomorphism types of pointed $s$-neighborhoods that occur in any  $A_{r,t}$ (equivalently, $B_{r,t}$). 
 For each $1\leq j\leq k$, let $A^j_{r,t}$ be the set of $a\in A_{r,t}$ such that $\gcb{A_{r,t}}{a}{s}$ is of isomorphism type $\iota_j$ and in like manner define $B^j_{r,t}$. It is easy to construct sequences $A_n=A_{r_n,t_n}$ and $B_n=B_{r_n,t_n}$ with $r_n<r_{n+1}$ and $t_n<t_{n+1}$, for all $n\in\omega$ with the following property:
\begin{dispar}\label{common-rare-dispar}
there is a partition of $\{1,\ldots,k\}$ into a pair of sets $U$ and $V$ such that
\begin{enumerate}
\item\label{rare}
 for every $j\in U$ there is an $m_j < t_1$ such that for all $n\geq 1$, $\card{A^j_n}=m_j=\card{B^j_n}$ and 
\item\label{com}
 for all $j\in V$ and for all $n\geq 1$, $\card{A^j_n},\card{B^j_n}>n$.
\end{enumerate}

\end{dispar}
We call the isomorphism type $\iota_j$ \emph{rare}, if it falls under \ref{common-rare-dispar}(\ref{rare}), and in this case, we also call points in $A^j_n$ and $B^j_n$ \emph{rare}; we call the isomorphism type $\iota_j$ \emph{common}, if it falls under \ref{common-rare-dispar}(\ref{com}), and in this case, we also call points in $A^j_n$ and $B^j_n$ \emph{common}. First, we choose $r'$ large enough so that at most a fixed number $t_{rare}$ of disjoint $r'$-neighborhoods suffices to cover all rare points in $A_n$ and $B_n$ for all $n$. It is easy to see that any $r' > 2s\cdot\sum_{j\in U}m_j$ suffices to guarantee this. Moreover, recalling that for all $n$, $\thr{A_n}{B_n}{r_n}{t_n}$ if $r_{n'}\geq r'$ and $t_{n'} \geq t_{rare}$, the disjoint $r'$-neighborhoods
$\gcb{A_{n'}}{a_{n'1}}{r'},\ldots,\gcb{A_{n'}}{a_{n't_A}}{r'}$ and $\gcb{B_{n'}}{a_{n'1}}{r'},\ldots,\gcb{B_{n'}}{a_{n't_B}}{r'}$, may be chosen so that $t_A=t_B$, and for every $1\leq i\leq t_A$, $\gcb{A_{n'}}{a_{n'i}}{r'}\cong \gcb{B_{n'}}{b_{n'i}}{r'}$.
It follows immediately, by the fact that \pres{P}\ is a local presentation scheme, that there are for each such $n'$, presentations $A^*_{n'}\pres{P}A_{n'}$ and  $B^*_{n'}\pres{P}B_{n'}$ such that for every $1\leq i\leq t_A$, $\rest{A^*_{n'}}{\gcb{A_{n'}}{a_{n'i}}{r'}}\cong \rest{B^*_{n'}}{\gcb{B_{n'}}{b_{n'i}}{r'}}$. We let $r^*=2s\cdot\sum_{j\in U}m_j$, and $q^*$ be the maximum size of an $r^*$-neighborhood in a graph of degree $d$. Then $q=q^*\cdot t_{rare}$ is the maximum size of a cover of the $s$-neighborhoods generated by $U$ by a system of disjoint $r^*$-neighborhoods. For all large enough $n$, we select $\alpha_n\subseteq A_n$ and $\beta_n\subseteq B_n$ to be such covers.
We now turn to deal with common points.

First, observe that for all $j\in V$, there are finitely many, say $j^*$, isomorphism types $\iota_{j1},\ldots,\iota_{jj^*}$ such that every presentation of an $s$-neighborhood of isomorphism type $\iota_j$ is of type $\iota_{ji}$, for some $1\leq i\leq j^*$. Let $j^{max}=\max{\{j^*\mid j\in V\}}$. Let $m= q + t_{d^*}\cdot j^\max$. We are now ready to apply Lemma \ref{wide-lemma}. First, let $p=\card{V}$ and $\zeta_s$ be the $p,q$-width growth factor for $\{G_C\mid C\in\strs{K}^d\}$ with respect to $s$. It follows, that if  
$\card{A^j_n},\card{B^j_n}>\zeta_s(m)$, then there are systems of points $Y^A_n=\{a_{ji}\in A^j_n\mid j\in V \mbox{ and } 1\leq i \leq t_{d^*}\cdot j^*\}$ $s$-scattered away from from $\alpha_n$ and $Y^B_n=\{b_{ji}\in B^j_n\mid j\in V \mbox{ and } 1\leq i \leq t_{d^*}\cdot j^*\}$ $s$-scattered away from from $\beta_n$.  
Now choose $n\geq\zeta_s(m)$, so that for every $j\in V$, $\card{A^j_n},\card{B^j_n}>\zeta_s(m)$ and thus $Y^A_n$ is $s$-scattered away from from $\alpha_n$ and $Y^B_n$ is $s$-scattered away from from $\beta_n$. We may now conclude, using disjoint local amalgamation for \pres{P}\ that there are $A^*\pres{P}A_n$ and $B^*\pres{P}B_n$ such that multiplicity of isomorphism types of presentations of each rare type in $A^*$ and $B^*$ are identical, and are all covered by $\alpha_n$ and $\beta_n$ respectively, and the multiplicity of isomorphism types of presentations of each common type in both $A^*$ and $B^*$ exceeds the threshold $t_{d*}$. Thus, $\thr{A^*}{B^*}{r_\theta}{t_{d^*}}$, which contradicts our hypothesis.
\end{proof}

The following theorem establishes non-uniform Hanf threshold locality for boolean queries invariantly first-order definable over bounded degree finite structures (as opposed to a arbitrary bounded degree structures) with respect to local, neighborhood-bounded elementary presentation schemes. We omit its proof, since it is virtually identical to the proof of Theorem \ref{locpresthreshloc-thm}.
\begin{theorem}\label{finlocpresthreshloc-thm}
Suppose  $\strs{K}\subseteq \strs{F}_\tau$ is closed under substructures and  \pres{P} is a local, neighborhood-bounded elementary $R$-presentation scheme for $\strs{K}$. If $\strs{J}\subseteq \strs{F}_\tau$ is a \pres{P}-invariant elementary boolean query on $\strs{K}$, then for all $d$, there are $r,t\in\omega$ such that $\strs{J}$ is Hanf $r,t$-threshold local on $\strs{K}\cap\strs{F}^d_\tau$.  
\end{theorem}

\section{Conclusion}

In this paper we establish three Hanf locality results for invariantly definable boolean queries, Theorems \ref{nbpresinvHanf-thm}, \ref{locpresthreshloc-thm}, and \ref{finlocpresthreshloc-thm}. These results suggest interesting avenues for further investigation. We list a few below.
\begin{enumerate}
\item 
Does Theorem \ref{locpresthreshloc-thm} (and \ref{finlocpresthreshloc-thm})  hold uniformly with respect to degree? We conjecture that the answer is no, that is, for some 
substructure closed class $\strs{K}\subseteq\strs{K}_\tau (\strs{F}_\tau)$ and some  \pres{P}\ local, neighborhood-bounded elementary presentation scheme for $\strs{K}$, there is a \pres{P}-invariant elementary boolean query $\strs{J}$ on $\strs{K}$ such that then for all $r$, there is a $d$, such that for all $t$, there are $A,B\in\strs{K}^d_\tau (\strs{F}^d_\tau)$ such that $\thr{A}{B}{r}{t}$, but $A\not\equiv_{\mathfrak{J}}B$.  
\item
Can Theorem \ref{finlocpresthreshloc-thm} be extended to other ``tame'' classes of structures, such as structures with Gaifman graphs of bounded tree-width. We are optimistic, insofar as application of generalizations of the notion of wide class, as deployed in the proof of Theorem \ref{finlocpresthreshloc-thm}, have been successfully applied to lift other model-theoretic results about degree-bounded structures to very general tame classes of finite structures (see \cite{ADK06,ADG08,Dawar10}).
\item 
In connection with Theorem \ref{nbpresinvHanf-thm}, it would be nice to have further examples of non-elementary queries that are elementarily invariantly definable via neighborhood bounded presentations over locally finite structures.
\end{enumerate}

\renewcommand{\baselinestretch}{1.0}
\bibliography{imfmt}

\end{document}